\documentclass[11pt,reqno]{amsart}
\usepackage[margin=1in]{geometry} 
\usepackage{color}
\usepackage{fixltx2e}
\usepackage{graphics} 
\usepackage{epsfig} 
\usepackage{amsmath} 
\usepackage{amssymb}  
\usepackage{amsthm}
\usepackage{todonotes}
\usepackage[norelsize,ruled,vlined]{algorithm2e}
\usepackage{graphicx}
\PassOptionsToPackage{usenames,dvipsnames,svgnames}{xcolor}  
\usepackage{tikz}
\usetikzlibrary{arrows,positioning,automata}

\newtheorem{theorem}{Theorem}
\newtheorem{remark}{Remark}
\newtheorem{lemma}{Lemma}
\newtheorem{definition}{Definition}

\let\oldtodo\todo
\renewcommand{\todo}[1]{\oldtodo{\footnotesize{#1}}}

\title[Decentralized Observability with Limited Communication between Sensors]{Decentralized Observability with \\Limited Communication between Sensors}

\author[A. B. Alexandru, S. Pequito, A. Jadbabaie, G. J. Pappas]{Andreea B. Alexandru $^{\dag}$ $\quad$ S\'ergio Pequito $^{\dag}$ $\quad$ Ali Jadbabaie  $^{\ddagger}$ \\George J. Pappas $^{\dag}$}

 \thanks{This work was supported in part by the TerraSwarm Research Center, one of six centers supported by the STARnet phase of the Focus Center Research Program (FCRP) a Semiconductor Research Corporation program sponsored by MARCO and DARPA}
\thanks{
$^{\dag}$Department of Electrical and Systems Engineering, School of Engineering and Applied Science,
\mbox{University of Pennsylvania}
\newline \indent
$^{\ddagger}$Institute for Data, Systems, and Society, Massachusetts Institute of Technology}

\date{} 

\begin{document}

\begin{abstract}
In this paper, we study the problem of jointly retrieving the state of a dynamical system, as well as the state of the sensors deployed to estimate it. We assume that the sensors possess a simple computational unit that is capable of performing simple operations, such as retaining the current state and model of the system in its memory. 

We assume the system to be observable (given all the measurements of the sensors), and we ask whether each sub-collection of sensors can retrieve the state of the underlying physical system, as well as the state of the remaining sensors. To this end, we consider communication between neighboring sensors, whose adjacency is captured by a communication graph. We then propose a linear update strategy that encodes the sensor measurements as states in an augmented state space, with which we provide the solution to the problem of retrieving the system and sensor states.

%
%


The present paper contains three main contributions. First, we provide necessary and sufficient conditions to ensure observability of the system and sensor states from any sensor. Second, we address the problem of adding communication between sensors when the necessary and sufficient conditions are not satisfied, and devise a strategy to this end. Third, we extend the former case to include different costs of communication between sensors. Finally, the concepts defined and the method proposed are used to assess the state of an example of approximate structural brain dynamics through linearized measurements.

%
\end{abstract}

\maketitle


\section{INTRODUCTION}

In the last decade, a significant effort was placed in developing strategies that enable the recovery of the system state, i.e., the problem of \emph{estimating} the system's state. The applications of such mechanisms are essential and include, among others, the monitoring of the electric power grid, process control, swarms of robots, and social networks~\cite{GuptaStateEstPowerSys,interconnectedDyn,Garin2010}. Furthermore, being able to retrieve the state of the system enables the assessment of the overall behavior of the plant, and allows us to design control strategies that enable the proper control of the system. Such control strategies can either steer the system to a specific target, as in a swarm of robots trying to keep a formation~\cite{MesbahiEgerstedt}, or stabilize the system, as in the case of the electric power grid whose frequency should be kept within a given range~\cite{ilic1998power}.

The estimation strategies can be broadly classified into \emph{centralized}, \emph{decentralized} and \emph{distributed}. In centralized schemes, it is assumed that all the data collected by the sensors deployed in the plant is available to a central entity that, together with the system's model, performs the estimation of the system's state. Observability of the system is commonly sought to ensure that it is possible to design stable estimators~\cite{Luenberger66}. Whereas the system might be observable when all sensors deployed in the system are considered, the same does not necessarily hold true when only a sub-collection of the sensors is considered. As  a consequence, we need to enhance the classical schemes with strategies that only consider a subset of sensors, as in the case of decentralized estimation~\cite{IKEDA83}. Moreover, we need to add communication between different sensor locations to obtain additional information about the system state, as in distributed estimation~\cite{Polychronopoulos}. For instance, the sensors can average either the estimates of the state obtained by the different sensors~\cite{khanPhd} or the innovations~\cite{DasMoura}, i.e, the residual between predicted and measured output, which usually leads to a smaller amount of information exchanged between sensors. 

In this paper, we propose a \emph{\mbox{distributed-decentralized}} scheme, i.e., a decentralized estimation in the sense previously defined, which resorts to communication in a similar fashion as distributed scenarios. Although it combines elements from both decentralized and distributed approaches, the \mbox{distributed-decentralized} approach is distinct from them in the following sense: it encodes the sensor measurements as states in an augmented state space that considers the physical system's dynamics and the dynamics induced by the communication between the sensors. Consequently, the sensors only need to share their state instead of estimates of the system's state or innovations. Moreover, the measured output of each sensor includes the measurements performed over the system's states, as well as the states of the remaining sensors it communicates with. 

This paper was developed in the context of large distributed systems with wireless sensors, where it is more convenient to add communication links rather than terminals to the system. Our method is also a \mbox{one-time} offline step, so that reaching consensus online in the context of distributed estimation is avoided, and the problem of not finishing computations before the following time step is bypassed.

\subsection*{Related Work}

Distributed and decentralized estimation approaches have been an active research topic in the last decade. Some of the most recent developments are overviewed, for instance, in~\cite{guptaPhd,karPhd,khanPhd,Subbotin}. 
More specifically, the problem of designing communication networks to solve dynamic estimation problems has been previously addressed in~\cite{Mohammadreza}, where sufficient conditions are provided in terms of the communication structure and classification of the different agents (sensors/state variables) in the system. 

 In \cite{PequitoC8,PequitoJ5}, the topology of a static network of sensors is designed in order to minimize the transmission cost among sensors and from the sensors to a central authority, allowing centralized field reconstruction. More recently,  \cite{PequitoECC14} addresses the problem of determining the minimum communication topologies to ensure observability of a multi-agent's network, given a potential communication graph. In~\cite{KhanJadbabaie}, some strategies based in consensus-like methods for distributed estimation are provided by resorting to structural systems theory. In addition,~\cite{Pajic} addresses the stabilization of a wireless control network with strategies from structural systems theory applied to augmented state systems.
 
The main contributions of this paper are threefold. First, we provide necessary and sufficient conditions to ensure \mbox{distributed-decentralized} observability. Second, we devise a method that guarantees observability when the necessary and sufficient conditions are not satisfied, by adding communication links between sensors. This implies providing a constructive algorithm that solves a maximum matching minimum cost problem. Third, we extend the former case to include different costs of communication between sensors.

The rest of the  paper is organized as follows. In Section~\ref{probStatement}, we provide the formal problem statement. Section~\ref{prelim} reviews graph theoretical concepts used in structural systems theory. Section~\ref{mainresults}  presents the main technical results, and discusses the computational complexity of the strategies proposed. Furthermore, the case study in Section \ref{illustrativeexample} illustrates the main results in the context of the brain dynamics. Conclusions and discussions on further research are presented in Section \ref{conclusions}.

\section{PROBLEM STATEMENT}\label{probStatement}

Consider a linear time-invariant system described by 
\begin{equation}
\mathbf x[k+1]=\mathbf A\mathbf x[k], \quad k=0,1,\ldots
\label{dynamics}
\end{equation}
where $x\in\mathbb{R}^{n\times 1}$ is the system's state, and $A\in\mathbb{R}^{n\times n}$ the matrix that determines the autonomous dynamics of the system. In addition, consider $m$ sensors deployed, whose measurements are described by
\begin{equation}
y_i[k]=\mathbf c_i \mathbf x[k], \quad i=1,\ldots,m,
\label{output}
\end{equation}
where $y_i\in \mathbb{R}$ is the measured output, and $\mathbf c_i\in\mathbb{R}^{1\times n}$ the output vector that encodes the linear combination of the states measured by the sensor $i$. 

In addition, we consider that \eqref{dynamics}-\eqref{output} is observable, but perhaps not when only some subset of sensors is considered, so that decentralized estimation is not guaranteed. Each sensor is equipped with a computational unit that is capable of performing elementary operations: it contains enough memory to retain the state estimates of the system and of the sensors, and it is capable of communicating with other sensors. These assumptions are common (although sometimes implicit) in distributed estimation.

Let $\mathcal G=(\mathcal V,\mathcal E)$ be the directed \emph{communication graph} that encodes the interactions between sensors, where $\mathcal V=\{1,\ldots,m\}$ identify the $m$ sensors described in~\eqref{output}, and an edge $(i,j)\in\mathcal E$, which we refer to as communication links, if sensor $j$ transmits to sensor $i$. Notice that we do not require the communication to be reciprocal between sensors. We consider that each sensor possesses a scalar state $z_i$, with $i=1,\ldots, m$, and  its evolution over time is assumed to be described as a linear combination of its previous state, the measured output of the system and the incoming states from neighboring sensors, i.e.,
\begin{equation}
z_i[k+1]=w_{ii}z_i[k]+y_i[k]+\sum_{j\in\mathbb N_i^-}w_{ij}z_j[k], \ i\in\mathcal{V},
\label{dynSensors}
\end{equation}
where $\mathbb N_i^-=\{j\in\mathcal{V}: (i,j)\in \mathcal E\}$ are the indices of the \mbox{in-neighbors} of sensor $i$ given by the communication graph~$\mathcal G$. 

Therefore, the overall dynamics described by~\eqref{dynamics}-\eqref{dynSensors} can be re-written as a linear augmented system:
\begin{equation}
\tilde{\mathbf x}[k+1]=\underbrace{\left[\begin{array}{cc}\mathbf A & \mathbf{0}_{n\times m}\\ \mathbf C & \mathbf W(\mathcal G) \end{array}\right]}_{\tilde{\mathbf A}(\mathcal G)}\tilde{\mathbf x}[k],
\label{dynAugmented}
\end{equation}
where $\tilde{\mathbf x}=[x_1\ \ldots\  x_n \ z_1 \ldots z_m]^\intercal$ is the augmented system's state, $\mathbf C=[\mathbf c_1^\intercal\ \mathbf c_2^\intercal \ \ldots \ \mathbf c_m^\intercal]^\intercal$ the measured output and $\mathbf W(\mathcal G)$ the dynamics between sensors induced by the communication graph, i.e., $[\mathbf W(\mathcal G)]_{ij}=w_{ij}$ when $(i,j)\in \mathcal E$ and zero otherwise. We consider that each sensor has access to its own state, thus, $\mathbf W(\mathcal G)$ has a non-zero diagonal. Subsequently, the output measured by the sensors is given by~\eqref{output} and the incoming states from the neighboring sensors, since the average rule in~\eqref{dynSensors} is performed at the sensor level, i.e., in its computational unit. Thus, the measured output for the augmented system is given by
\begin{equation}
\tilde y_i[k]=\underbrace{\left[\begin{array}{cc} -\  \mathbf c_i\  - & \mathbf{0}_{1\times m}\\ 
\mathbf{0}_{|\mathbb N_i^-|\times n} &\mathbb{I}_{m}^{\mathbb N_i^-}\end{array}\right]}_{\tilde{\mathbf C}_i} \tilde{\mathbf x}[k], \ i\in\mathcal V,
\label{outputAugmented}
\end{equation}
where $\mathbb{I}_{m}^{\mathcal J}$ is the sub-matrix containing the rows of the $m\times m$ identity matrix with indices in $\mathcal J\subset\{1,\ldots,m\}$.

In this setup, we aim to ensure that each sensor's computational unit is capable of retrieving the state of the augmented system. In other words,~\eqref{dynAugmented}-\eqref{outputAugmented} (or, equivalently $(\tilde{\mathbf A}(\mathcal G),\tilde{\mathbf C}_i)$), is observable for $i=1,\ldots,m$. Subsequently, we aim to address the following three related problems.

\emph{Problem~1} Determine the necessary and sufficient conditions that $\mathbf W(\mathcal G)$ must satisfy to ensure observability of $(\tilde{\mathbf A}(\mathcal G),\tilde{\mathbf C}_i)$ for $i=1,\ldots,m$. \hfill $\circ$


\emph{Problem~2} If the necessary and sufficient conditions to attain observability of $(\tilde{\mathbf A}(\mathcal G),\tilde{\mathbf C}_i)$ for $i=1,\ldots,m$ do not hold, then determine the minimum number of additional communication links which yield such conditions. \hfill $\circ$

In practice, the sensors might be deployed at varying distances from one another, making it more convenient to add some communication links instead of others. In other words, we can associate costs to the potential communication links from sensor $j$ to sensor $i$, which we denote by $\gamma_{ij}$, which is zero if the communication link already exists. Thus, we consider we are given the cost of setting a communication link between any two sensors and we pose the following problem:

\emph{Problem~3} If the necessary and sufficient conditions to attain observability of $(\tilde{\mathbf A}(\mathcal G),\tilde{\mathbf C}_i)$ for $i=1,\ldots,m$ do not hold, then determine the minimum number of additional communication such that the cost is minimized and observability holds. \hfill $\circ$

\section{PRELIMINARIES AND TERMINOLOGY}\label{prelim}

The following standard terminology and notions from structural systems theory and graph theory can be found, for instance, in \cite{PequitoJournal}. 
Structural systems deal with the sparsity (i.e., location of zeroes and non-zeroes) patterns of matrices, rather than with the numerical values of the elements. 
Let $\bar{\mathbf A}\in \{0,*\}^{n\times n}$ be the matrix that represents the structural pattern of $\mathbf A$ with the following encoding: if $\bar{\mathbf A}_{ij}=0$, then $\mathbf A_{ij} = 0$ and if $\bar{\mathbf A}_{ij} = *$, where $*$ is an arbitrary non-specified value $*$, then $\mathbf A_{ij}$ can take any value. Following the sparsity pattern, we associate structural matrices to every matrix in~\eqref{dynamics}-\eqref{output},~\eqref{dynAugmented}-\eqref{outputAugmented}, that will be employed further. 

A pair $(\bar{\mathbf A},\bar{\mathbf C})$ is structurally observable if and only if there exists an observable pair $(\mathbf A,\mathbf C)$ with the same sparseness as $(\bar{\mathbf A},\bar{\mathbf C})$. Moreover, given a structurally observable pair $(\bar{\mathbf A},\bar{\mathbf C})$, then \emph{almost all} pairs of real matrices $(\mathbf A,\mathbf C)$ with the same structure as $(\bar{\mathbf A},\bar{\mathbf C})$ are observable~\cite{Shields_Pearson:1976}.

Let $\mathcal D(\bar{\mathbf A})=(\mathcal X,\mathcal E_{\mathcal X,\mathcal X})$ be the digraph representation of $\bar{\mathbf A}$,  to be referred to as the \emph{state digraph},  where the vertex set $\mathcal X$ represents the set of state variables (also referred to as state vertices) and $\mathcal E_{\mathcal X,\mathcal X}=\{(x_i,x_j): \ \bar{\mathbf A}_{ji}\neq 0\}$ denotes the set of edges.  Similarly, we define the \emph{state-output digraph} $\mathcal D(\bar{\mathbf A},\bar{\mathbf C})=(\mathcal X\cup \mathcal Y,\mathcal E_{\mathcal X,\mathcal X}\cup\mathcal E_{\mathcal X,\mathcal Y})$, where $\mathcal Y$ represents the set of output vertices and $\mathcal E_{\mathcal X,\mathcal Y}=\{(x_j,y_i):\ \bar{\mathbf C}_{ji}\neq 0\}$.

A digraph $\mathcal{D}=(\mathcal V,\mathcal E)$ is \emph{strongly connected} if there exists a directed path between any  pair of vertices, i.e., a sequence of edges that starts and ends in those vertices. A sub-graph $\mathcal D_s=(\mathcal V_s,\mathcal E_s)$, with $\mathcal V_s\subset \mathcal V$ and $\mathcal E_s\subset \mathcal E$, is a \emph{strongly connected component (SCC)} if between any two vertices in $\mathcal V_s$ there exists a directed path, and $\mathcal D_s$ is maximal among such subgraphs. Visualizing each SCC as a virtual node (or supernode), one may generate a \textit{directed acyclic graph} (DAG), in which each node corresponds to a single SCC and a directed edge exists between two SCCs \emph{if and only if} there exists a directed edge connecting vertices in the SCCs in the original digraph. In the DAG representation, if the SCC does not have an edge from any of  its states to the states of another SCC, then it is referred to  as \emph{sink SCC}. Similarly, an SCC that does not have any incoming edge from other SCC is called a \emph{source SCC}.

For any two vertex sets $\mathcal S_{1}, \mathcal S_{2}\subset \mathcal V$, we define the \textit{bipartite graph} $\mathcal{B}(\mathcal S_1,\mathcal S_2,\mathcal E_{\mathcal S_1,\mathcal S_2})$ associated with $\mathcal D=(\mathcal V,\mathcal E)$, to be a directed graph, whose vertex set is given by $\mathcal S_{1}\cup \mathcal S_{2}$ and the edge set $\mathcal E_{\mathcal S_1,\mathcal S_2}$ by $ \mathcal E_{\mathcal S_1,\mathcal S_2}=\{(s_1,s_2)\in \mathcal E \ :\ s_1 \in \mathcal S_1, s_2 \in \mathcal S_2 \ \}$. 
Further, we refer to the bipartite graph $\mathcal B(\bar{\mathbf A})\equiv \mathcal B(\mathcal X,\mathcal X,\mathcal E_{\mathcal X,\mathcal X})$ as the \emph{state bipartite graph}, and to the $\mathcal B(\bar{\mathbf A},\bar{\mathbf C})\equiv \mathcal B(\mathcal X,\mathcal X\cup\mathcal Y,\mathcal E_{\mathcal X,\mathcal X\cup \mathcal Y})$ as the \emph{state-output bipartite graph}. Given $\mathcal{B}(\mathcal S_1,\mathcal S_2,\mathcal E_{\mathcal S_1,\mathcal S_2})$, a matching $M$ corresponds to a subset of edges in $\mathcal E_{\mathcal S_1,\mathcal S_2}$ that do not share vertices. In addition, we refer to the edges in a maximum matching as \emph{matching edges}. A maximum matching $M^{\ast}$ may then be defined as a matching $M$ that has the largest number of edges among all possible matchings. Such a matching decomposes the digraph into a disjoint set of cycles and elementary paths. The term \emph{\mbox{left-unmatched} vertices} (with respect to a maximum matching $M^{\ast}$ associated to $\mathcal{B}(\mathcal S_1,\mathcal S_2,\mathcal E_{\mathcal S_1,\mathcal S_2})$) refers to those vertices in $\mathcal S_{1}$ that do not have an outcoming matching edge in $M^{\ast}$.  For simplicity, we often just say \emph{a set of \mbox{left-unmatched} vertices}, omitting the explicit reference to the maximum matching when the context is clear. On the contrary, if we need to emphasize that the set of \mbox{left-unmatched} vertices is associated with a specific maximum matching  $M$ of the state bipartite graph, then we make the dependency explicit by $\mathcal U_L(M)$. Given a vector of weights $w$ associated to the edges in $\mathcal E$, we define a \emph{weighted matching} as the matching with weights corresponding to its constituting edges. Moreover, a \emph{minimum cost maximum matching (MCMM)} is a maximum matching $M^{\ast}$ whose edges achieve the minimum cost among all maximum matchings. Finally, we notice that there may exist several sets of \mbox{left-unmatched} vertices, since the maximum matching $M^{\ast}$ may not be unique. 

 \begin{center}
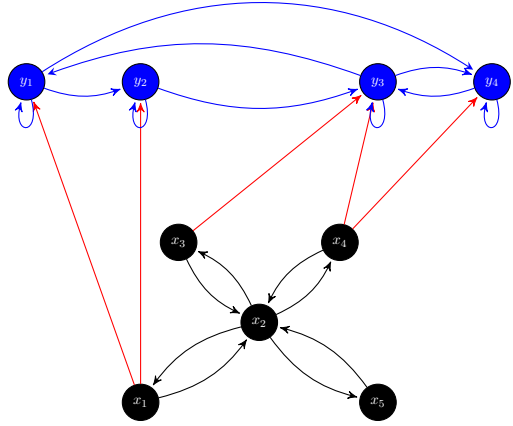
\begin{figure}
\begin{tikzpicture}[>=stealth',shorten >=1pt,node distance=1.5cm,on grid,initial/.style    ={}]

  \node[state,scale=0.5,fill]          (x2)                        {\textcolor{white}{$x_2$}};
  \node[state,scale=0.5,fill]          (x3) [above left =of x2]    {\textcolor{white}{$x_3$}};
  \node[state,scale=0.5,fill]          (x1) [below left =of x2,xshift=-1cm]    {\textcolor{white}{$x_1$}};
  \node[state,scale=0.5,fill]          (x4) [above right =of x2]    {\textcolor{white}{$x_4$}};
  \node[state,scale=0.5,fill]          (x5) [below right =of x2,xshift=1cm]    {\textcolor{white}{$x_5$}};
  \node[state,scale=0.5,fill=blue]          (y2) [above =of x1,yshift=5.5cm]    {\textcolor{white}{$y_2$}};
  \node[state,scale=0.5,fill=blue]          (y1) [left =of y2]    {\textcolor{white}{$y_1$}};
  \node[state,scale=0.5,fill=blue]          (y3) [above =of x5,yshift=5.5cm]    {\textcolor{white}{$y_3$}};
  \node[state,scale=0.5,fill=blue]          (y4) [right =of y3]    {\textcolor{white}{$y_4$}}; 
\tikzset{mystyle/.style={->,red}} 
\path (x1)     edge [mystyle]    node   {} (y1)
	(x1)     edge [mystyle]    node   {} (y2)
	(x3)     edge [mystyle]    node   {} (y3)
	(x4)     edge [mystyle]    node   {} (y3)
	(x4)     edge [mystyle]    node   {} (y4);

%

  \path[->]          (x1)  edge   [bend right=20]   node {} (x2);
  \path[->]          (x2)  edge   [bend right=20]   node {} (x1);
  \path[->]          (x3)  edge   [bend right=20]   node {} (x2);
  \path[->]          (x2)  edge   [bend right=20]   node {} (x3);
  \path[->]          (x4)  edge   [bend right=20]   node {} (x2);
  \path[->]          (x2)  edge   [bend right=20]   node {} (x4);
  \path[->]          (x5)  edge   [bend right=20]   node {} (x2);
  \path[->]          (x2)  edge   [bend right=20]   node {} (x5);
  \path[->,blue]          (y1)  edge   [bend right=20]   node {} (y2);
  \path[->,blue]          (y2)  edge   [bend right=20]   node {} (y3);
  \path[->,blue, -stealth, shorten >=.5mm]          (y3)  edge   [bend right=20]   node {} (y1);
  \path[->,blue]          (y3)  edge   [bend left=20]   node {} (y4);
  \path[->,blue]          (y4)  edge   [bend left=20]   node {} (y3);
  \path[->,blue, -stealth, shorten >=.5mm]          (y1)  edge   [bend left=33]   node {} (y4);
  \path[->,blue]          (y1)  edge   [loop below]   node {} (y1);
  \path[->,blue]          (y2)  edge   [loop below]    node {} (y2);
  \path[->,blue]          (y3)  edge   [loop below]    node {} (y3);
  \path[->,blue]          (y4)  edge   [loop below]    node {} (y4);
      
\end{tikzpicture}
  \caption{Plant with 5 state nodes (black) and 4 sensors nodes deployed (blue). The interconnections between the state nodes are depicted in black, the communication links 
  between the sensor nodes are depicted in blue, and the measurement terminals, between the state nodes and the sensor nodes, are depicted in red.}
  \label{systemExample}
\end{figure}
\end{center}

Figure~\ref{systemExample} provides an example of system that is observable from the whole set of sensors, but is not observable from any subset of sensors. More specifically, the system is not observable from sensors $1$ and $2$.  

The structural matrices of the system are the following:

 \[
\mathbf A =\left[
\begin{matrix}
0 & 1 & 0 & 0 & 0\\
1 & 0 & 1 & 1 & 1\\
0 & 1 & 0 & 0 & 0\\
0 & 1 & 0 & 0 & 0\\
0 & 1 & 0 & 0 & 0
\end{matrix}
\right],
\mathbf C =\left[
\begin{matrix}
1 & 0 & 0 & 0 & 0\\
1 & 0 & 0 & 0 & 0\\
0 & 0 & 1 & 1 & 0\\
0 & 0 & 0 & 1 & 0
\end{matrix}
\right],
\mathbf W(\mathcal{G}) =\left[
\begin{matrix}
1 & 0 & 1 & 0\\
1 & 1 & 0 & 0\\
0 & 1 & 1 & 1\\
1 & 0 & 1 & 1
\end{matrix}
\right].
\]

\begin{figure}
\begin{tikzpicture}[>=stealth',shorten >=1pt,node distance=0.5cm,on grid,initial/.style    ={}]

  \node[state,scale=0.3,fill]          (x1)                        {\textcolor{white}{$x_1$}};
  \node[state,scale=0.3,fill]          (x2) [below =of x1]    {\textcolor{white}{$x_2$}};
  \node[state,scale=0.3,fill]          (x3) [below =of x2]   {\textcolor{white}{$x_3$}};
  \node[state,scale=0.3,fill]          (x4) [below =of x3]    {\textcolor{white}{$x_4$}};
  \node[state,scale=0.3,fill]          (x5) [below =of x4]    {\textcolor{white}{$x_5$}};
  \node[state,scale=0.3,fill]          (z1) [below =of x5]    {\textcolor{white}{$z_1$}};
  \node[state,scale=0.3,fill]          (z2) [below =of z1]    {\textcolor{white}{$z_2$}};
  \node[state,scale=0.3,fill]          (z3) [below =of z2]    {\textcolor{white}{$z_3$}};
  \node[state,scale=0.3,fill]          (z4) [below =of z3]    {\textcolor{white}{$z_4$}}; 
  
  \node[state,scale=0.3,fill]          (x12) [right =of x1,xshift=6cm]  {\textcolor{white}{$x_1$}};
  \node[state,scale=0.3,fill]          (x22) [below =of x12]    {\textcolor{white}{$x_2$}};
  \node[state,scale=0.3,fill]          (x32) [below =of x22]   {\textcolor{white}{$x_3$}};
  \node[state,scale=0.3,fill]          (x42) [below =of x32]    {\textcolor{white}{$x_4$}};
  \node[state,scale=0.3,fill]          (x52) [below =of x42]    {\textcolor{white}{$x_5$}};
  \node[state,scale=0.3,fill]          (z12) [below =of x52]    {\textcolor{white}{$z_1$}};
  \node[state,scale=0.3,fill]          (z22) [below =of z12]    {\textcolor{white}{$z_2$}};
  \node[state,scale=0.3,fill]          (z32) [below =of z22]    {\textcolor{white}{$z_3$}};
  \node[state,scale=0.3,fill]          (z42) [below =of z32]    {\textcolor{white}{$z_4$}}; 
  
\tikzset{mystyle/.style={}} 
\path (x1)     edge [mystyle]    node   {} (x22)
	(x2)     edge [mystyle]    node   {} (x12)
	(x2)     edge [mystyle]    node   {} (x32)
	(x2)     edge [mystyle]    node   {} (x42)
        (x3)     edge [mystyle]    node   {} (x22)
	(x4)     edge [mystyle]    node   {} (x22)
	(x1)     edge [mystyle]    node   {} (z22)
	(x4)     edge [mystyle]    node   {} (z32)
	(z1)     edge [mystyle]    node   {} (z12)
	(z1)     edge [mystyle]    node   {} (z22)
	(z1)     edge [mystyle]    node   {} (z42)
	(z2)     edge [mystyle]    node   {} (z32)
	(z3)     edge [mystyle]    node   {} (z12)
	(z3)     edge [mystyle]    node   {} (z32)
	(z3)     edge [mystyle]    node   {} (z42)
	(z4)     edge [mystyle]    node   {} (z32)
	(z4)     edge [mystyle]    node   {} (z42);
	
\tikzset{mystyle/.style={red}} 
\path (x4)     edge [mystyle]    node   {} (z42)
	(x5)     edge [mystyle]    node   {} (x22)
	(x2)     edge [mystyle]    node   {} (x52)
	(x3)     edge [mystyle]    node   {} (z32)
	(x1)     edge [mystyle]    node   {} (z12)
	(z2)     edge [mystyle]    node   {} (z22);
	
\end{tikzpicture}
  \caption{Example of maximum matching in $\mathcal B(\tilde{\mathbf A}(\mathcal G))\equiv\mathcal B(\mathcal V\equiv(\mathcal X\cup \mathcal Z),\mathcal V,\mathcal E_{\mathcal V,\mathcal V})$, depicted by the collection of red edges.}  
   \label{maxMatching}
\end{figure}
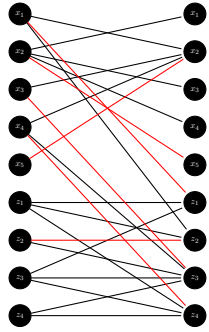

The concept of maximum matching stated above is exemplified in Figure~\ref{maxMatching}, where an instance of $M^{\ast}$ in $\mathcal B(\mathcal V\equiv(\mathcal X\cup \mathcal Z),\mathcal V,\mathcal E_{\mathcal V,\mathcal V})$ is shown, with the set of \mbox{left-unmatched} vertices being $\mathcal U_L(M^{\ast}) = \{z_1,z_3,z_4\}$.

Next, we state some known results regarding structural observability. The following theorem is an extension of the dual of Theorem 5 and Theorem 7 in~\cite{PequitoJournal} regarding the necessary and sufficient conditions for structural observability, applied to the system with the deployed sensors.

\begin{theorem}\label{structobs}
Let $\mathcal{D}(\bar{\mathbf A},\bar{\mathbf C}) = (\mathcal{X}\cup\mathcal{Y},\mathcal{E_{X,X}}\cup\mathcal{E_{X,Y}})$ denote the state-output digraph and $\mathcal{B}(\bar{\mathbf A}, \bar{\mathbf C})\equiv\mathcal{B}(\mathcal{X},\mathcal{X}\cup \mathcal{Y},\mathcal{E_{X,X\cup Y}})$ the state-output bipartite representation. The pair $(\bar{\mathbf A},\bar{\mathbf C})$ is structurally observable if and only if the following two conditions hold:
\begin{enumerate}
	\item[(i)] there is a path from every state vertex to an output vertex in $\mathcal D(\bar{\mathbf A}, \bar{\mathbf C})$; and
	\item[(ii)] there exists a maximum matching $M^{\ast}$ associated to $\mathcal{B}(\bar{\mathbf A}, \bar{\mathbf C})$ such that $\mathcal U_L (M^{\ast}) = \emptyset$. \hfill $\diamond$
\end{enumerate}
\end{theorem}

Justification: Theorem 7 in~\cite{PequitoJournal} states that a system is structurally observable if the left-unmatched vertices corresponding to a maximum matching associated to the \mbox{state-output} bipartite representation are measured by distinct sensors, and, at least one state from each sink SCC is measured by a sensor. In Theorem~\ref{structobs}, we define the \mbox{state-output} bipartite graph by adding the edges from the state vertices to the output vertices, which are the sensor nodes, to the state bipartite graph. Thus, we can pair the \mbox{left-unmatched} vertices in the system without sensors with their associated distinct sensors, which gives $\mathcal U_L (M) = \emptyset$. Moreover, since the sink SCCs are connected to a sensor, then every state vertex is the root of a path to an output vertex.

The next result allows us to interpret  a maximum matching of the state-output bipartite graph in terms of a representation in the state-output digraph.

\begin{lemma}[\hspace{-0.01cm}\cite{PequitoJournal}]\label{matMatPathCycle}
    Consider the digraph $\mathcal D(\bar{\mathbf A},\bar{\mathbf C})=(\mathcal X\cup \mathcal Y,\mathcal E_{\mathcal X,\mathcal X\cup\mathcal Y})$ and let $M^{\ast}$ be a maximum matching associated with the state-output bipartite graph $\mathcal B(\bar{\mathbf A},\bar{\mathbf C})\equiv \mathcal B(\mathcal X,\mathcal X\cup\mathcal Y,\mathcal E_{\mathcal X,\mathcal X\cup\mathcal Y})$. Then, the digraph $\mathcal D=(\mathcal X, M^{\ast})$  comprises a disjoint union of cycles and elementary paths, from the right-unmatched vertices to the \mbox{left-unmatched} vertices of $M^{\ast}$, that span $\mathcal D(\bar{\mathbf A},\bar{\mathbf C})$. Moreover, such a decomposition is \emph{minimal}, in the sense that no other spanning subgraph decomposition of $\mathcal D(\bar{\mathbf A},\bar{\mathbf C})$ into elementary paths and cycles contains strictly fewer elementary paths.
    \hfill $\diamond$
\end{lemma}

\section{MAIN RESULTS}\label{mainresults}

In this section, we present the main results of the present paper. 
To give a solution to \emph{Problem~1}, we propose to decouple the structure of the problem from its numeric parametrization.  First, we rely on structural system theory to ensure the necessary conditions, and then we show that the necessary conditions on the structure are sufficient to ensure observability through parametrization of the solution. These results are described in Theorem~\ref{decstructobs} and Theorem~\ref{sufficiency}, respectively.


Furthermore, we investigate the cases when the conditions for \emph{Problem~1} do not hold, as stated in \emph{Problem~2} and \emph{Problem~3}. The computational complexity of \emph{Problem~2} is given in Theorem~\ref{NP}. Because this is NP-hard, we provide an approximate polynomial constructive solution to \emph{Problem~2}, by describing an algorithm that decides which edges have to be added to ensure \mbox{distributed-decentralized} structural observability. The proposed procedure is provided in Algorithm~\ref{alg} and its correctness and computational complexity are asserted in Theorem~\ref{thmAlg}. Subsequently, in~\emph{Problem~3}, we address the situation of different costs for the communication links between sensors, which is also NP-hard, and whose approximate solution is similar to that of \emph{Problem~2}.

\subsection{Solution to Problem~1}

First, we ensure that \emph{Problem~1} can be solved when structural observability is sought.

\begin{theorem}\label{decstructobs}
Let $\mathcal D(\tilde{\mathbf A}(\mathcal G))=(\mathcal V\equiv(\mathcal X\cup \mathcal Z),\mathcal E_{\mathcal V,\mathcal V})$ be the state digraph, where $\mathcal X$ corresponds to the labels of the state vertices and $\mathcal Z$ to the labels of the sensors' states. In addition, let $\mathcal N_{i}^{-}=\{v\in\mathcal V: (v,z_i)\in\mathcal E\}$ be the set of in-neighbors of a vertex $z_i$ representing a sensor in $\mathcal D(\tilde{\mathbf A}(\mathcal G))$, $i = 1,\ldots,m$. The following two conditions are necessary and sufficient to ensure that~$(\tilde{\mathbf A}(\mathcal G),\tilde{\mathbf C}_i)$, for $i=1,\ldots,m$, is generically observable: 
\begin{enumerate}
\item[(i)] for every $z\in \mathcal Z$ there must exist a directed path from any $v\in\mathcal V$;
\item[(ii)]  for every $z\in \mathcal Z$ there must exist a set of \mbox{left-unmatched} vertices $\mathcal U_L$, associated with a maximum matching of the bipartite representation of $\mathcal D(\tilde{\mathbf A}(\mathcal G))$, such that  $\mathcal U_L\subset \mathcal N_i^-$ and $\mathcal U_L\cap \mathcal X = \emptyset$. \hfill 
$\diamond$
\end{enumerate}
\end{theorem}

\begin{proof}

We will prove that the above conditions are both necessary and sufficient for any subset of sensors deployed to the system, and, in particular, for every sensor $i$, with $i = 1,\ldots,m$. The proof will alternate between the state digraph $D(\tilde{\mathbf A}(\mathcal G))$ introduced above, and the state-output digraph $D(\tilde{\mathbf A}(\mathcal G),\tilde{\mathbf C}_i)$, introduced in Theorem~\ref{structobs}. To avoid burdensome notation, we will use $\tilde{\mathbf A}$ instead of $\tilde{\mathbf A}(\mathcal G)$ in the following.

\emph{Necessity}: We will start by assuming that $(\tilde{\mathbf A},\tilde{\mathbf C}_i)$, for $i=1,\ldots,m$, is structurally observable for any subset of sensors in $\mathcal Z$ and prove that conditions~(i) and~(ii) hold. To this end, we have leveraged Theorem~\ref{structobs} for a convenient application. Let $\mathbb N_i^-=\{j\in\{1,\ldots,m\}: (i,j)\in \mathcal E\}$ be the set of indices of the in-neighbors of sensor $i$ given by the communication graph $\mathcal G$ (notice the difference between $\mathbb N_i^-$ and $\mathcal N_i^-$). Let us define $\mathcal Y=\{y_1,\ldots y_m\}$ as the set of outputs of the augmented system. These are virtual outputs that are attached to the sensors' states and are defined by $\mathbb{I}_m^{\mathbb N_i^-}$, for each $i=1,\ldots,m$. Let $\mathcal{D}(\tilde{\mathbf A}, \tilde{\mathbf C}) = (\mathcal V\cup\mathcal Y,\mathcal{E_{V,V\cup Y}})$ be the state-output digraph of the augmented system. Condition~(i) of Theorem~\ref{structobs} applied to $\mathcal D(\tilde{\mathbf A}, \tilde{\mathbf C})$ gives that every output of the system is reachable from any vertex in $\mathcal V$ and condition~(i) of Theorem~\ref{decstructobs} ensues. 

Let $\mathcal{B}(\tilde{\mathbf A}, \tilde{\mathbf C}_i)\equiv\mathcal{B}(\mathcal V,\mathcal V\cup\mathcal Y,\mathcal{E_{V,V\cup Y}})$ be the augmented system's digraph. Condition~(ii) of Theorem~\ref{structobs} for $\mathcal D(\tilde{\mathbf A},\tilde{\mathbf C}_i)$ gives that there is no \mbox{left-unmatched} vertex in $\mathcal V$ with respect to a maximum matching in $\mathcal{B}(\tilde{\mathbf A}, \tilde{\mathbf C}_i)$. Therefore, the only possible \mbox{left-unmatched} vertices in the bipartite state graph $\mathcal{B}(\tilde{\mathbf A})\equiv\mathcal{B}(\mathcal V,\mathcal V,\mathcal{E_{V,V}})$ are $z\in \mathcal Z$. This gives $\mathcal U_L\cap \mathcal X = \emptyset$. Subsequently, the remaining \mbox{left-unmatched} vertices $z\in\mathcal Z$ are measured by the outputs $y\in\mathcal Y$ given by $\mathbb{I}_n^{\mathbb N_i^-}$. Hence, the \mbox{left-unmatched} vertices $\mathcal U_L$ have to be in $\mathcal N_i^-$ by the construction of $\tilde{\mathbf C}_i$, and condition~(ii) also holds.



\emph{Sufficiency}: To prove that the conditions in Theorem~\ref{decstructobs} are sufficient, we consider~(i) and~(ii) satisfied, and show that the augmented system is structurally observable.

We consider the digraph $\mathcal D(\tilde{\mathbf A}, \tilde{\mathbf C}_i)$ and the associated state-output bipartite graph $\mathcal B(\tilde{\mathbf A},\tilde{\mathbf C}_i)$. As the name ``sink" suggests, we can regard the augmented system as a layered system, with the system's states in the top layer, followed by the layer of sensors' states, and the system outputs layer as the bottom layer. The structural observability of~\eqref{dynamics}-\eqref{output} implies that the sink SCCs in the augmented system are in the layer of sensors. Hence, there is a path from every vertex in $\mathcal V$ to an output vertex in $\mathcal D(\tilde{\mathbf A},\tilde{\mathbf C}_i)$ and condition~(i) of Theorem~\ref{structobs} is granted.

Condition~(ii) in Theorem~\ref{decstructobs} states that there is a maximum matching in the state bipartite graph $\mathcal B(\tilde{\mathbf A})$ such that no state vertex is a \mbox{left-unmatched} vertex and all \mbox{left-unmatched} vertices are in-neighbors of the sensor $i$. Since there are edges from the sensors to the output vertices in the augmented system represented by  $\mathcal D(\tilde{\mathbf A}, \tilde{\mathbf C}_i)$, there is a maximum matching $M^{\ast}$ in $\mathcal B(\tilde{\mathbf A},\tilde{\mathbf C}_i)$ such that no system state and no sensor state are \mbox{left-unmatched}. In other words, $\mathcal U_L (M^{\ast}) = \emptyset$, so condition~(ii) of Theorem~\ref{structobs} is fulfilled. This completes the proof of the structural observability of $(\tilde{\mathbf A},\tilde{\mathbf C}_i)$ for $i=1,\ldots,m$.
\end{proof}


Structural observability holds for almost all numerical realizations of both $\tilde{\mathbf A}(\mathcal G)$ and of $\tilde{\mathbf C}_i$. Nonetheless, we notice that, unlike~\cite{dionSurvey}, we want the parametrization of $(\mathbf A,\mathbf C)$ to be fixed. Therefore, we show that with the observability of $(\mathbf A,\mathbf C)$ and with a parametrization of $\mathbf W(\mathcal G)$ alone, it is possible to attain the observability of $(\tilde{\mathbf A}(\mathcal G), \tilde{\mathbf C}_i)$.

\begin{theorem}\label{almostAll}
If $(\mathbf A,\mathbf C)$ is observable and $(\tilde{\mathbf A}(\mathcal G),\tilde{\mathbf C}_i)$, for $i=1,\ldots,m$, is structurally observable, then almost all realizations of $\mathbf W(\mathcal G)$ ensure that $(\tilde{\mathbf A}(\mathcal G),\tilde{\mathbf C}_i)$ is observable for $i=1,\ldots,m$.
\hfill $\diamond$
\label{sufficiency}
\end{theorem}

\begin{proof}

We make the same notational assumptions as in the proof of Theorem~\ref{decstructobs}.
We resort to the well known Popov-Belevitch-Hautus observability criterion \cite{Hespanha09}. The pair $(\tilde{\mathbf A},\tilde{\mathbf C}_i)$ is observable if and only if 
\[
\text{rank}\underbrace{\left [ \begin{array}{c}
\tilde{\mathbf A}-\lambda \mathbb{I}_{n+m}\\
\tilde{\mathbf C}_i
\end{array}\right ]}_{M}=n+m, \ \lambda \in\mathbb{C}, \ i=1,\ldots,m.
\]
Using \eqref{dynAugmented} and \eqref{outputAugmented}, $M$ can be further written as follows:
\[\begin{array}{ll}
M&=\left[
\begin{matrix}
\mathbf A-\lambda \mathbb{I}_{n} &\mathbf{0}_{n\times m}\\
\mathbf C& \mathbf W(\mathcal G)-\lambda \mathbb{I}_{m}\\ 
-\  \mathbf c_i\  - & \mathbf{0}_{1\times m}\\ 
\mathbf{0}_{|\mathbb N_i^-|\times n} &\mathbb{I}_{m}^{\mathbb N_i^-}\end{matrix}
\right].
\end{array}
\]

Now, notice that as a consequence of condition (i) in Theorem~\ref{decstructobs}, $\mathcal G$ is a strongly connected graph. Furthermore, a sensor is assumed to be able to communicate with itself, therefore, $\mathbf W(\mathcal G)$ is irreducible and its diagonals can be arbitrarily chosen. In particular, for almost all possible parameterizations of $\mathbf W(\mathcal G)$, we can choose the weights such that $\mathbf W(\mathcal G)$ has only simple eigenvalues that are distinct from those of $\mathbf A$, and $\text{rank}(\mathbf W(\mathcal G))=n$. Thus, $\lambda$ can be in the spectrum of $\mathbf A$ or in the spectrum of $\mathbf W(\mathcal G)$. We will treat both these cases in the following.

(i) $\lambda\notin \text{spec}(\mathbf A)$. Then,
\[
\text{rank}(\mathbf M)=\underbrace{\text{rank} (\mathbf A-\lambda \mathbb{I}_n)}_{=n} + \underbrace{\text{rank} \left[ \begin{array}{c}
 \mathbf W(\mathcal G)-\lambda \mathbb{I}_m\\
 \mathbb{I}_m^{\mathbb N_i^-}
\end{array}\right]}_{=m}.
\]
The first term equals to $n$ because $\lambda\notin \text{spec}(\mathbf A)$ does not produce a collapse of rank in $\mathbf A-\lambda \mathbb{I}_n$. $(\mathbf W(\mathcal G),\mathbb{I}_m^{\mathbb N_i^-})$ is structurally observable, since $\mathbf W(\mathcal G)$ is irreducible and has \mbox{zero-free} diagonal.
For almost all parameterizations of $\mathbf W(\mathcal G)$, $(\mathbf W(\mathcal G),\mathbb{I}_m^{\mathbb N_i^-})$ is observable. Therefore, by choosing such a parametrization and applying Popov's criterion, we obtain that the second term equals~$m$.

(ii) $\lambda \in \text{spec}(\mathbf A)$. Thus, $\lambda \notin \text{spec}(\mathbf W(\mathcal G))$ for almost all chosen parameterizations. Using Schur's complement\footnote{
For the matrices $\mathbf A\in\mathbb R^{n\times m}, \mathbf B\in\mathbb R^{n\times n},\mathbf D\in\mathbb R^{p\times n}$, the following is true:
\[ \left[ \begin{matrix}\mathbf A & \mathbf B\\ \mathbf 0_{p\times m} &\mathbf D \end{matrix}\right] = 
\left[ \begin{matrix} \mathbb I_n & \mathbf 0_{n\times p}\\\mathbf D\mathbf B^{-1} & \mathbb I_{p} \end{matrix}\right] 
\left[ \begin{matrix} \mathbf 0_{n\times m} & \mathbf B\\ -\mathbf D\mathbf B^{-1}\mathbf A& \mathbf 0_{p\times  m}\end{matrix}\right] 
\left[ \begin{matrix} \mathbb I_{m} & \mathbf 0_{m\times n} \\\mathbf B^{-1}\mathbf A & \mathbb I_{n} \end{matrix}\right].
\]
Since the transformations are full-rank, then $ \text{rank}\left[ \begin{matrix}\mathbf A & \mathbf B\\ \mathbf 0_{p\times m} &\mathbf D \end{matrix}\right] = \text{rank}
\left[ \begin{matrix} \mathbf 0_{n\times m} & \mathbf B\\ -\mathbf D\mathbf B^{-1}\mathbf A& \mathbf 0_{p\times  m}\end{matrix}\right].$
},
 the expression for the rank of $M$ can be \mbox{re-written} as follows:
\[
\text{rank}(\mathbf M)=\text{rank}\left[
\begin{matrix}
\mathbf A-\lambda \mathbb{I}_{n} &\mathbf{0}_{n\times m}\\
-\  \mathbf c_i\  - & \mathbf{0}_{1\times m}\\ 
\mathbf{0}_{m\times n}& \mathbf W(\mathcal G)-\lambda \mathbb{I}_{m}\\
\mathbb{I}_{m}^{\mathbb N_i^-}(\mathbf W(\mathcal G)-\lambda \mathbb{I}_{m})^{-1}C&\mathbf{0}_{|\mathbb N_i^-| \times m}
\end{matrix}
\right].
\]
Let $\mathbf H_i(\mathcal G) = \mathbb{I}_{m}^{\mathbb N_i^-}(\mathbf W(\mathcal G)-\lambda \mathbb{I}_{m})^{-1}\mathbf C$. $\mathbf H_i(\mathcal G)$ corresponds to the transfer function from the sensors identified by $\mathbf C$ to the measurements identified by $\mathbb{I}_m^{\mathbb N_i^-}$. Thus, 
\[
\text{rank}(M)= \underbrace{\text{rank}(\mathbf W(\mathcal G)-\lambda \mathbb{I}_m)}_{=m}+ \text{rank} \underbrace{\left[ \begin{array}{c}
\mathbf A-\lambda \mathbb{I}_n\\
 \mathbf H_i(\mathcal G)\end{array}\right]}_{\mathbf M_A}.
\]

The weights in $\mathbf W(\mathcal G)$ can be designed such that the diagonal entries are distinct, and $\lambda \notin \text{spec}(\mathbf W(\mathcal G))$ gives full rank to the first term. From \cite{van1988graph}, $\mathbf H_i(\mathcal G)$ generates  almost all subspaces of dimension $d$, where $d$ corresponds to the number of vertex disjoint paths from the sensors identified by $\mathbf C$ to the measurements identified by $\mathbb{I}_m^{\mathbb N_i^-}$. Next, notice that the number of vertex disjoint paths $d$ is at least equal to the number of \mbox{left-unmatched} vertices of the original system, since these have to be in a neighborhood of each sensor of the augmented system, as prescribed 
in condition~(ii) in Theorem~\ref{decstructobs}. Therefore, we have that $\text{rank}(\mathbf M_A)=n$ for almost all realizations of $\mathbf W(\mathcal G)$, and the result follows.
\end{proof}

\subsection{Solution to Problem~2}\label{SubsP2}

Next, we leverage some of the conditions prescribed to have solutions to \emph{Problem 1}~to obtain a solution to \emph{Problem 2}. First, notice that determining a solution to \emph{Problem 2}~can be challenging since it requires to ensure two disjoint conditions, see Theorem~\ref{decstructobs}. In addition, it requires those conditions to hold across different sensors.

Therefore,  we first provide solutions to \emph{Problem 2}~when one of the conditions of Theorem~\ref{decstructobs} holds, hence, only the remaining condition needs to be ensured. More specifically, we have two scenarios:  (i) there are \mbox{left-unmatched} vertices in $\mathcal N_i^-$ for every $i$ such that $\mathcal U_L\subset \mathcal N_i^-$, and (ii)~$\mathcal G$~is strongly connected. The second scenario is equivalent to condition (i) in Theorem~\ref{decstructobs}, since we already know the sensors are reachable from every state $x\in\mathcal X$ due to the observability of~\eqref{dynamics}-\eqref{output}. Thus, we only need to ensure that each subset of sensors is reachable from the other subsets of sensors.

In the first scenario, we want to prove that $\mathcal G$ is strongly connected. Towards this goal we need to introduce the following definition~\cite{sofIP}:

\begin{definition}[Sequential-pairing]\label{seqPairing}
    Consider  two sets of indices $\mathcal I=\{i_1,\ldots,i_n\}$ and $\mathcal J=\{j_1,\ldots,j_n\}$, and  a maximum matching $M$ of the bipartite graph $\mathcal B(\mathcal J,\mathcal I,\mathcal E_{\mathcal J,\mathcal I})$, where  $\mathcal E_{\mathcal J,\mathcal I} \subseteq \mathcal J\times\mathcal I$. We denote by $|\mathcal I,\mathcal J\rangle_M$ a sequential-pairing induced by $M$, defined as follows:
    $$
        |\mathcal I,\mathcal J\rangle_M=\Bigg(\bigcup_{l=2,\ldots,k}\{(i_l,j_{l-1})\}\Bigg) \cup \big\{(i_1,j_k)\big\},
    $$
    where  $(i_l,j_l)\in M$, for $l=1,\ldots,k$ and $k = min\{|\mathcal I|,|\mathcal J|\}$. \hfill $\diamond$
\end{definition}

\begin{remark}
The sequential-pairing consists of the collection of edges such that $M\cup |\mathcal I,\mathcal J\rangle_M$ forms  a cycle. A graph spanned by a cycle formed by $M\cup |\mathcal I,\mathcal J\rangle_M$ is a strongly connected graph.~\hfill $\diamond$
\end{remark}

The communication graph $\mathcal G$ can be decomposed into strongly connected components, with each sink strongly connected component being reachable from at least one source strongly connected component. Subsequently, we have the following result:


\begin{lemma}\label{maxEdge}
Let the DAG decomposition of $\mathcal G$ be given by a collection of SCCs denoted by $\{\Sigma_i\}_{i=1,\ldots, N}$, where $\Sigma^{so}\equiv \{\Sigma_j\}_{j=1,\ldots, \alpha}$ are source SCCs and $\Sigma^{si}\equiv\{\Sigma_k\}_{k=1,\ldots,\beta}$ are sink SCCs, where $\alpha+\beta<=N$. In addition, consider $\mathcal B\equiv \mathcal B(\Sigma^{so},\Sigma^{si},\mathcal E_{\Sigma^{so},\Sigma^{si}})$, where $\mathcal E_{\Sigma^{so},\Sigma^{si}}=\{(\Sigma_j,\Sigma_k): \text{ if } \Sigma_k\text{ is reachable from }\Sigma_j \}$. Given a maximum matching $M^{\ast}$ of $\mathcal B$ and the sequential pairing induced by it, the minimum number of edges to be added to ensure $\mathcal G$ to be strongly connected is $\max\{\alpha,\beta\}$. \hfill $\diamond$
\end{lemma}

\begin{proof}
The outline of the proof is the following: construct a sequential-pairing, so that joined with the associated maximum matching, it gives a cycle that spans $\mathcal G$. 

Consider a maximum matching in $\mathcal B(\Sigma^{so},\Sigma^{si},\mathcal E_{\Sigma^{so},\Sigma^{si}})$. Then, add the edges such that the source SCCs are reachable from the sink SCCs, equivalent to the sequential pairing $|\mathcal J,\mathcal K\rangle_M$, where $\mathcal J=\{j_1,\ldots,j_n\}$ corresponding to $\Sigma^{si}$ and $\mathcal K=\{k_1,\ldots,k_n\}$ corresponding to $\Sigma^{so}$.

(i) If $\alpha=\beta$ and there is a disjoint path from every $\Sigma^{so}_i$ to $\Sigma^{si}_i$, then the step above is enough to build the cycle, and the number of edges added is equal to $\alpha=\beta$. If there are indices $i\in\mathcal{I}=\{i_1,\ldots,i_{\alpha}\}$ for which there is no disjoint path from $\Sigma^{so}_i$ to $\Sigma^{si}_i$, then after adding the $\alpha-|\mathcal I|$ edges of the sequential pairing, collapse the strongly connected components into new SCCs. Although some edges from the matching considered in Definition~\ref{seqPairing} are missing, there are other edges from $\Sigma^{so}_{i_l}$ to $\Sigma^{si}_k$ and from $\Sigma^{so}_j$ to $\Sigma^{si}_{i_l}$, where $i_l\in\mathcal I$, $j\in\mathcal J \setminus \mathcal I$ and $k\in\mathcal K\ \setminus\mathcal I$. Subsequently, add the $|\mathcal I|$ edges that complete the matching, therefore, completing the cycle.

(ii) If $\alpha>\beta$, then after adding the $\beta$ edges corresponding to the sequential pairing, collapse the graph into a new, extended SCC and $\alpha-\beta$ source SCCs. For each source SCC, add an edge from the extended SCC to make the graph strongly connected. Thus, the number of edges added is equal to $\alpha$.

(iii) If $\alpha<\beta$, then, similarly to (ii), collapse the graph into a new, extended SCC and $\beta-\alpha$ sink SCCs. From each sink SCC, add an edge to the extended SCC to make the graph strongly connected. The number of edges added in this case is equal to $\beta$.
\end{proof}

\begin{remark}
An isolated strongly connected component is both source and sink, therefore, the existence of such components increase both $\alpha$ and $\beta$. Hence, Lemma~\ref{maxEdge} also holds for graphs with isolated SCCs.\hfill $\diamond$
\end{remark}

The condition that $\mathcal G$ is strongly connected is a mild condition since many results in literature like~\cite{Rao91,KhanJadbabaie} assume that the communication graph is strongly connected from its construction. 

In the second scenario, we assume that the communication digraph is strongly connected, so one is required to ``bring'' the \mbox{left-unmatched} vertices to a neighborhood of the sensors' states, as required by condition (ii) in Theorem~\ref{decstructobs}. Recall that there might exist several possible sets of \mbox{left-unmatched} vertices associated with maximum matchings. In particular, some of these \mbox{left-unmatched} vertices might be either the dynamical system's state vertices or sensors' state vertices. Since we have assumed that the original system is observable, we are guaranteed that there exists a maximum matching such that only $\mathcal U_L\subset\mathcal Z$. However, in order to assess if the condition $\mathcal U_L\subset \mathcal N_i^-$ holds for a sensor $i$, we have to explore the matchings that produce no \mbox{left-unmatched} system's state vertices.

\begin{theorem}\label{NP}
	The problem of determining the minimum number of communication links that ensure the observability of the pair $(\tilde{\mathbf A}(\mathcal G),\tilde{\mathbf C}_i)$, for every $i=1,\ldots m$ is NP-hard.
	\hfill $\diamond$
\end{theorem}

\begin{proof}

\emph{Problem~2} is NP-hard since it consists in determining the minimum number of edges that need to be added to a digraph such that there exist $k$ disjoint paths between given subsets of vertices. A reference to the variant of this problem where the subsets always have the same cardinality and the number of disjoint paths is constant can be found in \cite{Kobayashi10}, and this variant is NP-hard. 
\end{proof}

\begin{remark}\label{left-unmatched} 
In practice, the total number of \mbox{left-unmatched} vertices is often $0$, $1$ or $2$~\cite{siljak2007large}. In the latter case, it is required to find two disjoint paths from the \mbox{left-unmatched} vertices of the \mbox{state-output} digraph to the \mbox{in-neighboring} vertices of the sensor's state, where one can often be obtained through the \mbox{state-output} digraph and the other through the communication digraph, if this is assumed strongly connected. \hfill $\diamond$
\end{remark}

As a result of Remark~\ref{left-unmatched}, we are able to propose computationally tractable solutions to \emph{Problem~2}. We need to produce a strategy that penalizes the \mbox{left-unmatched} vertices that are the system's state vertices, and favors those states in the neighborhood of a given sensor's state. To this effect, we employ the MCMM routine, which can be found in~\cite{cormenBook}, for example. We define $\mathcal G^{\ast}$ as the union between the communication graph $\mathcal G$ and the set of the communication links that have to be added to guarantee observability of the state-output digraph. The procedure is detailed in Algorithm~\ref{alg}. 

\begin{algorithm}\label{alg}
\small

\textbf{Input:}~$\mathcal D(\tilde{\mathbf A}(\mathcal G))=(\mathcal V\equiv(\mathcal X\cup \mathcal Z),\mathcal E_{\mathcal V,\mathcal V}), \tilde{\mathbf C}_i,i = 1,\ldots,m$;\\
\textbf{Output:}~ The modified communication graph $\mathcal G^{\ast}$ and $\mathbf W(\mathcal G^{\ast})$;\\

\For {$i=1,\ldots,m$} {
\textbf{1.} Let $\mathcal S=\{s_j: j\in\{1,\ldots,m\}\setminus \mathbb N_i^-\}$ be the set of slack variables corresponding to virtual outputs of the sensors' state, with $\mathcal E_{\mathcal Z,\mathcal S}=\{(z_j,s_j): j\in\{1,\ldots, m\}\setminus \mathbb N_i^-\}, z_j \in\mathcal Z, s_j \in\mathcal S\}$;\\
\textbf{2.} Let $\mathcal B(\tilde{\mathbf A}, \tilde{\mathbf C}_i, S)=(\mathcal V,\mathcal V\cup \mathcal S,$ $\mathcal{E_{V,V}}\cup\mathcal E_{\mathcal V,\mathcal S})$ be the state-output bipartite graph of the system with slack variables;\\
\textbf{3.} Let the weight function of the edges be $w:\mathcal{E}\rightarrow \mathbb{R}^+$, where $\mathcal E$ is the set of all edges. Assign the following costs $w^i$ for sensor~$i$: ~(i) $w^i(e) = 0$, $e \in \mathcal{E_{V,V}}$, and~(ii) $w^i(e) = 1$, $e\in\mathcal{E_{Z,S}}$;\\
\textbf{4.} Define $(\mathcal B(\tilde{\mathbf A}, \tilde{\mathbf C}_i, S),w)$ as the weighted state-output bipartite graph;\\
\textbf{5.} Run the MCMM on $(\mathcal B(\tilde{\mathbf A}, \tilde{\mathbf C}_i, S),w)$ and obtain $M^{\ast}$;\\
\textbf{6.} For all $j=1,\ldots,m$ such that $\{(z_j,s_j)\in M^{\ast}\setminus \mathcal{E_{Z,S}}\}$,  add $(z_i,z_j)$ to $\mathcal G$.\\
}
Let $\mathcal G^{\ast}=\mathcal G$;\\
Find $\mathbf W(\mathcal G^{\ast})$ that yields the result in Theorem~\ref{almostAll}.
\caption{Finding the communication graph that ensures observability of $(\tilde{\mathbf A}(\mathcal G), \tilde{\mathbf C}_i)$, given strong connectivity of $\mathbf W(\mathcal G)$.}
\end{algorithm}

Intuitively, the introduction of the slack variables $\mathcal S$ in Algorithm~\ref{alg} is connected to adding the communication links that ensure $\mathcal U_L\subset \mathcal N_i^-$. Any slack variable $s_j\in\mathcal U_L(M^{\ast})$ suggests adding a communication link from sensor $z_j$ to sensor $z_i$. 

\begin{theorem}\label{thmAlg}
The procedure outlined in Algorithm~\ref{alg} is correct, i.e., its execution ensures the observability of the augmented pair $(\tilde{\mathbf A}(\mathcal G),\tilde{\mathbf C}_i)$, for any sub-collection of sensors considered. In addition, the complexity of the algorithm is $\mathcal O(m(n+2m)^3)$.\hfill $\diamond$
\end{theorem}

\begin{proof}

A maximum matching of a bipartite graph decomposes the digraph into a disjoint union of a minimum collection of paths and cycles (see Lemma~\ref{matMatPathCycle}), where the \mbox{left-unmatched} vertices correspond to the ending vertices of the paths. The provided cost structure ensures that when a minimum weight maximum matching is computed, we use as many of the system's state vertices as possible for forming paths and cycles, since the total cost incurred by its edges is zero. Furthermore, the paths formed in the decomposition of the state digraph (that are not cycles) are extended by the edges in $\mathcal{E_{V,Z}}$ to the state-output digraph, with edges also of zero cost. Subsequently, there are two options. Option~(i): it is possible to continue adding edges of zero cost to those paths until we have the chance to include an edge that ends in a vertex in the neighborhood of sensor~$i$'s state, which minimizes the overall sum. Option~(ii): it is not possible to extend the path to contain the edges mentioned in~(i) and the end of the path is a slack variable, i.e., the sensor state corresponding to the slack variable is a \mbox{left-unmatched} vertex. If the mentioned sensor states are not in the in-neighborhood of sensor $i$, then we add edges in the communication graph (with associated weights in $\mathbf W(\mathcal G)$ and $w^i$) such that $\mathcal U_L\subset \mathcal N_i^-$. 

The computational complexity of Algorithm~\ref{alg} is that of the MCMM procedure, namely $\mathcal O(N^3)$, where $N$ is the maximum number of vertices between the partitions in the bipartite graph that gives the maximum matching~\cite{cormenBook}. In the present case, $N = n+m+p$ for each of the $m$ sensors, where $p = |\mathcal S|\leq m$. Since we are considering the worst case scenario, Algorithm~\ref{alg} has complexity $\mathcal O(m(n+2m)^3)$.
\end{proof}

In summary, the above strategy enables us to find the minimum number of \mbox{left-unmatched} vertices in the neighborhood of a sensor, hence, one can simply add communication links from the \mbox{left-unmatched} vertices that do not belong to the in-neighborhood to the state of sensor~$i$. Let us refer to the number of such \mbox{left-unmatched} vertices with respect to sensor~$i$ as $q_i$. Subsequently, we can simply perform the same procedure for each sensor, taking into consideration the edges added for the previous sensors, and we end up having a total of $q_1+\ldots + q_m$ additional links. This number is less than or equal to the number of communication links that would have to be added for each sensor, independently.

One can argue that this strategy is not necessarily optimal since the number of edges added is not minimal, and some of the edges added in application of the procedure for sensor~$i$ might be used in constructing a path in the MCMM in the procedure for sensor $i+1$. Therefore, one can ask if it is possible to solve this problem optimally. This problem is NP-hard, as shown in Theorem~\ref{NP}, but since, by Remark~\ref{left-unmatched}, the maximum number of \mbox{left-unmatched} vertices that can appear in practice is two, the difference between the number of edges added by Algorithm~\ref{alg} and the minimum number of edges that need to be added is likely to be negligible.

The communication graph $\mathcal G$ for the system in Figure~\ref{systemExample} is strongly connected, so the second condition of Theorem~\ref{decstructobs} holds. A maximum matching of minimum cost is illustrated in Figure~\ref{maxMatchingP2}. Following the algorithm for $i=1$, the communication link that has to be added is from sensor $4$ to sensor $1$. One parametrization for the communication matrix that ensures \mbox{distributed-decentralized} observability after running the algorithm for $i\in\{1,2,3,4\}$ is:
\[ \mathbf W(\mathcal{G}^{\ast}) =\left[
\begin{matrix}
0.9597    &     0  &  0.3404  &  0.5853\\
    0.2238 &   0.7513     &    0 &   0.2551\\
         0  &  0.5060 &   0.6991  &  0.8909\\
    0.9593    &     0  &  0.5472  &  0.1386
\end{matrix}
\right].
\]

\begin{center}
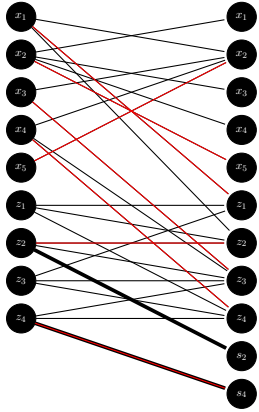
\begin{figure}
\begin{tikzpicture}[>=stealth',shorten >=1pt,node distance=0.5cm,on grid,initial/.style    ={}]

  \node[state,scale=0.4,fill]          (x1)                        {\textcolor{white}{$x_1$}};
  \node[state,scale=0.4,fill]          (x2) [below =of x1]    {\textcolor{white}{$x_2$}};
  \node[state,scale=0.4,fill]          (x3) [below =of x2]   {\textcolor{white}{$x_3$}};
  \node[state,scale=0.4,fill]          (x4) [below =of x3]    {\textcolor{white}{$x_4$}};
  \node[state,scale=0.4,fill]          (x5) [below =of x4]    {\textcolor{white}{$x_5$}};
  \node[state,scale=0.4,fill]          (z1) [below =of x5]    {\textcolor{white}{$z_1$}};
  \node[state,scale=0.4,fill]          (z2) [below =of z1]    {\textcolor{white}{$z_2$}};
  \node[state,scale=0.4,fill]          (z3) [below =of z2]    {\textcolor{white}{$z_3$}};
  \node[state,scale=0.4,fill]          (z4) [below =of z3]    {\textcolor{white}{$z_4$}}; 
  
  \node[state,scale=0.4,fill]          (x12) [right =of x1,xshift=6cm]  {\textcolor{white}{$x_1$}};
  \node[state,scale=0.4,fill]          (x22) [below =of x12]    {\textcolor{white}{$x_2$}};
  \node[state,scale=0.4,fill]          (x32) [below =of x22]   {\textcolor{white}{$x_3$}};
  \node[state,scale=0.4,fill]          (x42) [below =of x32]    {\textcolor{white}{$x_4$}};
  \node[state,scale=0.4,fill]          (x52) [below =of x42]    {\textcolor{white}{$x_5$}};
  \node[state,scale=0.4,fill]          (z12) [below =of x52]    {\textcolor{white}{$z_1$}};
  \node[state,scale=0.4,fill]          (z22) [below =of z12]    {\textcolor{white}{$z_2$}};
  \node[state,scale=0.4,fill]          (z32) [below =of z22]    {\textcolor{white}{$z_3$}};
  \node[state,scale=0.4,fill]          (z42) [below =of z32]    {\textcolor{white}{$z_4$}}; 
  \node[state,scale=0.4,fill]          (s22) [below =of z42]    {\textcolor{white}{$s_2$}};
  \node[state,scale=0.4,fill]          (s42) [below =of s22]    {\textcolor{white}{$s_4$}};
  
\tikzset{mystyle/.style={}} 

\path (x1)     edge [mystyle]    node   {} (x22)
	(x2)     edge [mystyle]    node   {} (x12)
	(x2)     edge [mystyle]    node   {} (x32)
	(x2)     edge [mystyle]    node   {} (x42)
        (x3)     edge [mystyle]    node   {} (x22)
	(x4)     edge [mystyle]    node   {} (x22)
	(x1)     edge [mystyle]    node   {} (z22)
	(x4)     edge [mystyle]    node   {} (z32)
	(z1)     edge [mystyle]    node   {} (z12)
	(z1)     edge [mystyle]    node   {} (z22)
	(z1)     edge [mystyle]    node   {} (z42)
	(z2)     edge [mystyle]    node   {} (z32)
	(z3)     edge [mystyle]    node   {} (z12)
	(z3)     edge [mystyle]    node   {} (z32)
	(z3)     edge [mystyle]    node   {} (z42)
	(z4)     edge [mystyle]    node   {} (z32)
	(z4)     edge [mystyle]    node   {} (z42)
	(z2)     edge [mystyle]    node   {} (s22)
	(x4)     edge [mystyle]    node   {} (z42)
	(x5)     edge [mystyle]    node   {} (x22)
	(x2)     edge [mystyle]    node   {} (x52)
	(x3)     edge [mystyle]    node   {} (z32)
	(x1)     edge [mystyle]    node   {} (z12)
	(z2)     edge [mystyle]    node   {} (z22)
	(z4)     edge [mystyle]    node   {} (s42)	;	

\tikzset{mystyle/.style={-,double=black}} 
\path 	(z4)     edge [mystyle]    node   {} (s42)	
	(z2)     edge [mystyle]    node   {} (s22);

\tikzset{mystyle/.style={,red}} 

\path (x4)     edge [mystyle]    node   {} (z42)
	(x5)     edge [mystyle]    node   {} (x22)
	(x2)     edge [mystyle]    node   {} (x52)
	(x3)     edge [mystyle]    node   {} (z32)
	(x1)     edge [mystyle]    node   {} (z12)
	(z2)     edge [mystyle]    node   {} (z22)
	(z4)     edge [mystyle]    node   {} (s42)	;
	
\end{tikzpicture}
  \caption{Minimum cost maximum matching in $(\mathcal B(\tilde{\mathbf A}(\mathcal G),\tilde{\mathbf C}_i,\mathcal S),w))$, for sensor 1, depicted by the collection of red edges. The bold edges have weight $1$, whereas all the rest have weight $0$.}
  \label{maxMatchingP2}
\end{figure}
\end{center}

\subsection{Solution to Problem~3}

The binary cost strategy designed for the solution of \emph{Problem~2} is expandable to variable costs corresponding to adding the communication links between the sensors. Suppose we are provided with a matrix of costs $\mathbf \Gamma$, with elements $\gamma_{ij}$ denoting the costs of adding a communication link from sensor $j$ to sensor $i$. Using these costs, we can adapt the previous strategy to the cost constrained problem of choosing the edges to be added such that the system is \mbox{distributed-decentralized} observable. The procedure implements the idea of including the costs $\gamma_{ij}$ in the weights for the bipartite representation $(\mathcal B(\tilde{\mathbf A}, \tilde{\mathbf C}_i, S),w)$.

\begin{theorem}\label{P3}
	Let $\{\gamma_{ij}\}_{i,j\in\{1,\ldots,m\}}$ be the communication cost, and consider the following weight structure: (i)~$w^i(e) = 0$, $e \in \mathcal E\cup \mathcal{E_{X,Y}}$, and (ii)~$w^i(e) = \gamma_{ij} $, $e=(z_j,s_j), z_j\in\mathcal Z\setminus \mathcal N_i^-, s_j \in \mathcal S$, for every $i,j = 1,\ldots,m$. Algorithm~\ref{alg} correctly computes and selects the communication links that have to be added to ensure the observability of the pair $(\tilde{\mathbf A}(\mathcal G),\tilde{\mathbf C}_i)$, when the cost of communication is imposed and the weight pattern $w^i$ for $i=1,\dots,m$ is considered. 
	\hfill $\diamond$
\end{theorem}

The proof is similar to the proof of Theorem~\ref{thmAlg}. As before, the algorithm for finding the MCMM will select the maximum number of edges of cost zero. Following the design of the new weighting pattern, when the communication links already available are not enough to satisfy $\mathcal U_L \subset \mathcal N_i^-$, the algorithm will try to minimize the cost of the links and not their number. 

In addition, observations regarding the minimum number of communication links can be made, in a similar manner as in the end of Section~\ref{SubsP2}.


\section{ILLUSTRATIVE EXAMPLE}\label{illustrativeexample}

In this section, we use the concepts and methodology proposed in Section~\ref{mainresults} in the context of brain dynamics. We consider a linearized brain dynamics associated with data obtained from electroencephalography (EEG)~\cite{Solovey2012}, and whose structure is induced by the brain structural connectivity obtained via magnetic resonance imaging~\cite{deReus2013397}. More specifically, consider the brain partitioned in $34$ different regions~\cite{Fair09}, and its connectivity captured by the digraph $\mathcal S=(\mathcal V=\{1,\ldots,34\}, \mathcal E)$, where each $v\in\mathcal V$ labels a different region and the edges in $\mathcal E$ capture the existence of white-matter tracks between two different regions, as illustrated in Figure~\ref{brain}. In addition, the activity in the different regions and between these is fluctuating, being more pronounced during the execution of certain tasks. Therefore, following~\cite{Solovey2012}, under the assumption that the activity is induced by the brain structure~\cite{Gu15}, the state evolution can be captured by considering the first-order autoregressive model:
\[
\mathbf x[k+1]=\mathbf A(\mathcal S)\mathbf x[k] + \mathbf \epsilon[k], \quad k=0,1,\ldots,
\]
where $\mathbf x\in\mathbb{R}^{34}$ is the state of the different regions, and $[\mathbf A_k(\mathcal S)]_{i,j}=0$ if $(j,i)\notin \mathcal E$, else, a scalar to be determined, and $\epsilon_k$ is the dynamics error. These scalars can be determined using least-square approach with sparsity inducing methods relying on basis-pursuit optimization~\cite{Figueiredo}, when some measurements collected by EEG are considered. Nonetheless, we assume these to be zero-one as described in~\cite{Fair09}, where additionally, the diagonal entries were set to zero and only $30\%$ of the \mbox{off-diagonal} entries were considered. The digraph $\mathcal D(\mathbf A)$ is composed of four SCCs as shown in Figure~\ref{brainGraph}, with two source SCCs containing the states numbered $21$ and $23$ and one sink SCC containing the state numbered $22$. In addition, there exists a maximum matching of the state bipartite graph with unmatched-vertices $\mathcal U_L=\{21,22\}$; hence, measuring these two variables suffices to ensure structural observability.
\begin{figure}
    \includegraphics[width=0.6\textwidth]{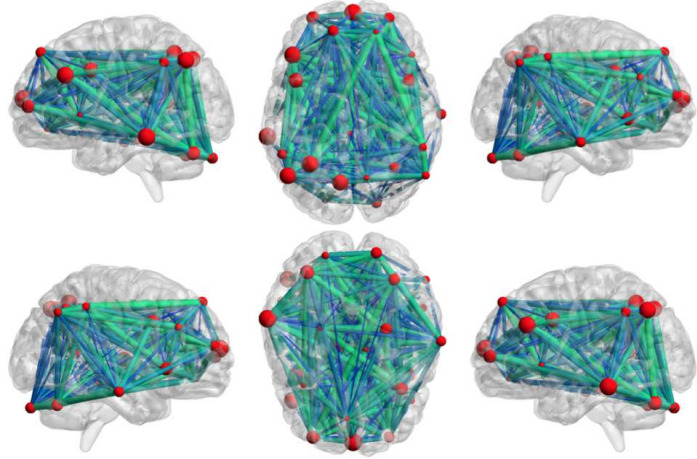}
  \caption{Full view of the brain with the 34 regions represented as nodes, and the white matter tracks represented as edges, obtained using \cite{BrainNet}.}      
   \label{brain}
\end{figure}

For the simplicity of the model (and reproducibility of the  results), we assume that an EEG sensor captures the behavior of a single region in a linearized manner, i.e.,
$y_k=\mathbb I^{\mathcal J}_{34}x_k$,
where $\mathcal J=\{21,22,23,24,29\}$ correspond to sensors deployed in the following scalp locations  \{AF3, AF4, T7, T8, Pz\} (see details of these locations in~\cite{Koessler09}), and consistent with the locations of the EMOTIV Insight~\cite{emotivInsight}. In addition, the system is observable -- in particular, we notice that considering that other regions were simultaneously measured by the EEG sensors would not compromise the observability under the present dynamics. Subsequently, considering the design topology of this device, we assume the following communication graph (through wiring) holds, where $\textbf{adj}(\mathcal G)$ is the adjacency matrix of graph $\mathcal G$:
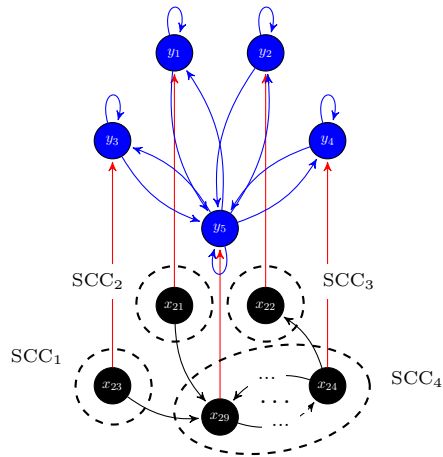
\begin{figure}
\begin{tikzpicture}[>=stealth',shorten >=1pt,node distance=2cm,on grid,initial/.style    ={}]

  \node[state,scale=0.5,fill=blue]         (y5)                        {\textcolor{white}{$y_5$}};
  \node[state,scale=0.5,fill=blue]         (y3) [above left =of y5,yshift=-0.5cm]    {\textcolor{white}{$y_3$}};
  \node[state,scale=0.5,fill=blue]          (y1) [above right =of y3,xshift=-1.2cm,yshift=-0.5cm]    {\textcolor{white}{$y_1$}};
  \node[state,scale=0.5,fill=blue]         (y4) [above right =of y5,yshift=-0.5cm]    {\textcolor{white}{$y_4$}};
  \node[state,scale=0.5,fill=blue]         (y2) [above left =of y4,xshift=1.2cm,yshift=-0.5cm]    {\textcolor{white}{$y_2$}};
  \node[state,scale=0.5,fill]          (x21) [below =of y1,yshift=-2.7cm]    {\textcolor{white}{$x_{21}$}};
  \node[state,scale=0.5,fill]          (x22) [below =of y2,yshift=-2.7cm]    {\textcolor{white}{$x_{22}$}};
  \node[state,scale=0.5,fill]          (x23) [below =of y3,yshift=-2.5cm]    {\textcolor{white}{$x_{23}$}};
  \node[state,scale=0.5,fill]          (x24) [below =of y4,yshift=-2.5cm]    {\textcolor{white}{$x_{24}$}}; 
  \node[state,scale=0.5,fill]          (x29) [below =of y5,yshift=-1cm]    {\textcolor{white}{$x_{29}$}}; 

\tikzset{mystyle/.style={->,red}} 
\path (x21)     edge [mystyle]    node   {} (y1)
	(x22)     edge [mystyle]    node   {} (y2)
	(x23)     edge [mystyle]    node   {} (y3)
	(x24)     edge [mystyle]    node   {} (y4)
	(x29)     edge [mystyle]    node   {} (y5);

  \path[->,blue]          (y1)  edge   [bend right=20]   node {} (y5);
  \path[->,blue]          (y5)  edge   [bend right=20]   node {} (y1);
  \path[->,blue]          (y2)  edge   [bend right=20]   node {} (y5);
  \path[->,blue]          (y5)  edge   [bend right=20]   node {} (y2);
  \path[->,blue]          (y3)  edge   [bend right=20]   node {} (y5);
  \path[->,blue]          (y5)  edge   [bend right=20]   node {} (y3);
  \path[->,blue]          (y4)  edge   [bend right=20]   node {} (y5);
  \path[->,blue]          (y5)  edge   [bend right=20]   node {} (y4);
  \path[->,blue]          (y1)  edge   [loop above]   node {} (y1);
  \path[->,blue]          (y2)  edge   [loop above]    node {} (y2);
  \path[->,blue]          (y3)  edge   [loop above]    node {} (y3);
  \path[->,blue]          (y4)  edge   [loop above]    node {} (y4);
  \path[->,blue]          (y5)  edge   [loop below]    node {} (y5);

  \path[->]          (x21)  edge   [bend right=20]   node {} (x29);
  \path[->]          (x23)  edge   [bend right=20]   node {} (x29);
  \path[->]          (x24)  edge   [bend right=20]   node {} (x22);
  \tikzset{every node/.style={fill=white}} 

  \path[->]          (x24)  edge   [bend right=35]   node {\tiny{...}} (x29);
  \path[->]         (x29)  edge   [bend right=35]   node {\tiny{...}} (x24);

\draw[thick,dashed] (-1.4,-2.1) circle (0.5cm);
\draw[thick,dashed] (-0.6,-1) circle (0.5cm);
\draw[thick,dashed] (0.6,-1) circle (0.5cm);
\draw[thick,dashed,rotate=10] (0.28,-2.35) ellipse (1.3cm and 0.7cm);

\node[text width=0.7cm] at (-2.4,-1.7) {\tiny SCC\textsubscript{1}};
\node[text width=0.7cm] at (-1.6,-0.7) {\tiny SCC\textsubscript{2}};
\node[text width=0.7cm] at (1.7,-0.7) {\tiny SCC\textsubscript{3}};
\node[text width=0.7cm] at (2.6,-2) {\tiny SCC\textsubscript{4}};
\node[text width=0.5cm] at (0.77,-2.3) {\small\dots};
\end{tikzpicture}
  \label{brainGraph}
 \caption{DAG representation of the 34 regions of the brain, along with the sensors deployed for measurement of the system. The strongly connected component labeled as SCC\textsubscript{4} contains all the other state vertices not depicted. }
\end{figure}

$\textbf{adj}(\mathcal G)= \left[ \begin{matrix} 1 & 0 & 0 & 0 & 1\\ 0 & 1 & 0 & 0 & 1\\ 0 & 0 & 1 & 0 & 1\\ 0 & 0 & 0 & 1 & 1\\ 1 & 1 & 1 & 1 & 1 \end{matrix} \right]$. 

Because the system is observable, it follows that it is also structurally observable, and, in particular, there are distinct sensors measuring the locations $\{21,22\}$. 
Furthermore, notice that there are direct paths from every vertex in the digraph to every sensor, which implies that condition (i) of Theorem~\ref{decstructobs} is satisfied. However, condition (ii) is not fulfilled, and therefore, we need to address \emph{Problem~3}, where we consider a unit of cost $c$ and a communication cost structure in accordance to the distance between the sensors, given by:

$\mathbf \Gamma= \left[ \begin{matrix} 0 & 2c & 3c & 4c & 4c\\ 2c & 0 & 4c & 3c & 4c\\ 3c & 4c & 0 & 5c & 3c\\ 4c & 3c & 5c & 0 & 3c\\ 4c & 4c & 3c & 3c & 0 \end{matrix} \right]$.

Running Algorithm~\ref{alg} with the given costs, as in Theorem~\ref{P3}, we obtain that there is only one communication link from sensor $2$ to sensor $4$ that has to be added, while incurring the minimum cost. The resulting communication scheme that guarantees observability of the system with respect to each sensor is $\mathcal G^{\ast}$, with a total cost of $3c$, and a possible communication protocol $\mathbf W(\mathcal G^{\ast})$, given as following: 

$\textbf{adj}(\mathcal G^{\ast})= \left[ \begin{matrix} 1 & 0 & 0 & 0 & 1\\ 0 & 1 & 0 & 0 & 1\\ 0 & 0 & 1 & 0 & 1\\ 0 & 1 & 0 & 1 & 1\\ 1 & 1 & 1 & 1 & 1 \end{matrix} \right]$ and $\mathbf W(\mathcal G^{\ast}) = \left[ \begin{matrix} 0.77 &  0 & 0 & 0 & 0.48\\ 0 & 0.43 & 0 & 0 & 0.44\\ 0 & 0 & 0.30 & 0 & 0.50\\ 0 & 0.51 & 0 & 1.00 & 0.81\\ 1.00 & 0.79 & 0.64 & 1.00 & 0.37 \end{matrix} \right]$. 

In conclusion, the system $(\tilde{\mathbf A}(\mathcal G),\tilde{\mathbf C}_i)$ as in~\eqref{dynAugmented}-\eqref{outputAugmented} is observable for all $i = 1,\ldots, m$.

\section{CONCLUSIONS AND FURTHER RESEARCH}\label{conclusions}

In this paper, we addressed the problem of exact \mbox{distributed-decentralized} retrieval of the states of an LTI system, from any subset of sensors of the given system. We considered that the sensors have storing and communicating capabilities, and that they are inter-linked according to a communication graph. Our approach involved associating states to the sensors and constructing an augmented system, for which we provided the necessary and sufficient conditions to ensure observability from any sensor. Furthermore, we addressed the problem of re-designing the communication graph to ensure  \mbox{distributed-decentralized} observability when the previous conditions are not readily fulfilled and proved that it is NP-hard. We devised a suboptimal solution that can be attained in polynomial time such that the observability requirements are satisfied. Moreover, we proposed an extension to the previous strategy that takes into consideration the variable costs of adding communication links between the sensors. We explored the trade-offs between the different aspects of the control-communication-computation paradigm for the present setup that employed more communication and a lighter computational load.


Future research includes identifying subclasses of the problem in which the solution described in the present paper is optimal and examining suboptimality guarantees. Moreover, we notice that, in some scenarios, adding communication capabilities may be more prohibitive than adding more memory to the sensors. We will further investigate the implications of having multi-dimensional sensors' states, and address the possibility of relying on the same communication graph and providing different communication schemes such that a \mbox{distributed-decentralized} scheme is feasible.





\footnotesize

\bibliographystyle{IEEEtran}
\bibliography{acc2016}

\end{document}